\documentclass[11pt]{amsart}
\usepackage{amsmath,amssymb,latexsym,soul,cite,mathrsfs}

\usepackage{color,enumitem,graphicx}
\usepackage[colorlinks=true,urlcolor=blue,
citecolor=red,linkcolor=blue,linktocpage,pdfpagelabels,
bookmarksnumbered,bookmarksopen]{hyperref}
\usepackage[english]{babel}
  \usepackage{graphicx}
  \usepackage[dvipsnames]{xcolor}
  \usepackage{trig}
  \usepackage{color}
  \usepackage{graphics}
  \definecolor{Mygrey}{gray}{0.75}

\usepackage[left=2.9cm,right=2.9cm,top=2.8cm,bottom=2.8cm]{geometry}
\usepackage[hyperpageref]{backref}

\usepackage[colorinlistoftodos]{todonotes}
\makeatletter
\providecommand\@dotsep{5}
\def\listtodoname{List of Todos}
\def\listoftodos{\@starttoc{tdo}\listtodoname}
\makeatother

\numberwithin{equation}{section}
\def\R {{\rm I}\hskip -0.85mm{\rm R}}
\def\N {{\rm I}\hskip -0.85mm{\rm N}}

\newtheorem{theorem}{Theorem}[section]
\newtheorem{proposition}[theorem]{Proposition}
\newtheorem{lemma}[theorem]{Lemma}

\newtheorem{remark}{Remark}

\title[Multiplicity for an equation with degenerate nonlocal diffusion]
{Multiplicity of positive solutions for an equation with degenerate nonlocal diffusion}

\author[L. Gas\'inski]{Leszek Gasi\'nski}
\author[J. R. Santos Jr.]{Jo\~ao R. Santos J\'unior}

\address[L. Gasi\'nski]{\newline\indent Faculty of Mathematics and Computer Science
\newline\indent
 Jagiellonian University
\newline\indent
 ul. Lojasiewicza 6
\newline\indent
 30-348 Krak\'ow, Poland}
\email{\href{mailto: Leszek.Gasinski@ii.uj.edu.pl}{Leszek.Gasinski@ii.uj.edu.pl}}

\address[J. R. Santos Jr.]{\newline\indent Faculdade de Matem\'atica
\newline\indent
Instituto de Ci\^{e}ncias Exatas e Naturais
\newline\indent
Universidade Federal do Par\'a
\newline\indent
Avenida Augusto corr\^{e}a 01, 66075-110, Bel\'em, PA, Brazil}
\email{\href{mailto: joaojunior@ufpa.br}{joaojunior@ufpa.br}}

\thanks{Leszek Gasi\'nski was partially supported by the National Science Center of Poland under Project No. 2015/19/B/ST1/01169.
        J. R. Santos J\'unior was partially supported by CNPq 302698/2015-9 and CAPES 88881.120045/2016-01, Brazil.}

\begin{document}

\maketitle
\begin{abstract}
Even without a variational background, a multiplicity result of positive solutions with ordered $L^{p}(\Omega)$-norms is provided to the following boundary value
problem
\begin{equation*}
\left \{ \begin{array}{ll}
-a(\int_{\Omega}u^{p}dx)\Delta u = f(u) & \mbox{in $\Omega$,}\\
u=0 & \mbox{on $\partial\Omega$,}
\end{array}\right.
\end{equation*}
where $\Omega$ is a bounded domain and $a$, $f$ are continuous real functions with $a$ vanishing in many positive points.
\end{abstract}
\maketitle

%------------------------------------------------------------------------------
\section{Introduction}

%------------------------------------------------------------------------------

We are going to investigate the existence of multiple positive solutions for the following class of degenerate nonlocal problems
\begin{equation}\label{P}\tag{P}
\left \{ \begin{array}{ll}
-a\left( \int_{\Omega}u^{p}dx\right)\Delta u = f(u) & \mbox{in $\Omega$,}\\
u>0 & \mbox{in $\Omega$,}\\
u=0 & \mbox{on $\partial\Omega$,}
\end{array}\right.
\end{equation}
where $\Omega$ is a bounded smooth domain in $\R^{N}$, $p\geq 1$, $a\in C([0,\infty))$ and $f\in C^{1}(\R)$ are functions which, in a first moment, verifies only:

\medskip

\begin{enumerate}
\item[(H0)]\label{a} there exist positive numbers $0=:t_{0}<t_{1}<t_{2}<\ldots<t_{K}$ ($K\ge 1$) and $t_{\ast}>0$ such that:

\smallskip

$a(t_{k})=0$, $a>0$ in $(t_{k-1}, t_{k})$, for all  $k\in\{1, \ldots, K\}$, $f(t)>0$ in $(0, t_{\ast})$ and $f(t_{\ast})=0$.
\end{enumerate}

\medskip

By considering the same sort of hypothesis (H0), authors in \cite{SS} (motivated by papers like \cite{AA}, \cite{BB} and \cite{Hess}) have proven that problem
\begin{equation}\label{Q}\tag{KP}
\left \{ \begin{array}{ll}
-a\left( \int_{\Omega}|\nabla u|^{2}dx\right)\Delta u = f(u) & \mbox{in $\Omega$,}\\
u>0 & \mbox{in $\Omega$,}\\
u=0 & \mbox{on $\partial\Omega$,}
\end{array}\right.
\end{equation}
has at least $K$ positive solutions whose $H_{0}^{1}(\Omega)$-norms are ordered, provided that an appropriated \emph{area condition} relating $a$ and $f$
holds. The approach used in \cite{SS} is strongly based on the variation structure of problem \eqref{Q}.

A natural and interesting question related to \eqref{P} to be addressed in this paper is: once broken variational structure of problem \eqref{Q} (by a change in the
type of nonlocal term), would still persist the existence of multiple ``ordered solutions''?

By using an approach completely different from \cite{SS}, present paper provides a positive answer to the last question under suitable assumptions on $a$ and
$f$. In order to state in a precise way our main result we need to introduce some assumptions:

\medskip

\begin{enumerate}
\item[(H1)]\label{f2} map $(0, t_{\ast})\ni t\mapsto f(t)/t$ is decreasing;
\end{enumerate}

\medskip

\begin{remark}\label{rem}
It follows from (H1) that $\gamma:=\lim_{t\to 0^{+}}f(t)/t$ is well defined and it can be a positive number or $+\infty$, depending on value assumed by $f$ at $0$. In fact, it is clear that if $f(0)=0$, then, since $f\in C^{1}(\R)$ and $\gamma=f'(0)$, we will have $\gamma<+\infty$. On the other hand, if $f(0)>0$ then $\gamma=+\infty$.
\end{remark}

\medskip

Before setting further assumptions let us introduce some notation. Along the paper, $\|\cdot\|$, $|\cdot|_{r}$, $\lambda_{1}$, $\varphi_{1}$ and $e_{1}$ denote
$H_{0}^{1}(\Omega)$-norm, $L^{r}(\Omega)$-norm, first eigenvalue of minus Laplacian with homogeneous Dirichlet boundary condition, positive eigenfunction
associated to $\lambda_{1}$ normalized in $H_{0}^{1}(\Omega)$-norm and positive eigenfunction associated to $\lambda_{1}$ normalized in $L^{\infty}(\Omega)$-norm,
respectively.

Our last assumptions relate functions $a$ and $f$.

\medskip

\begin{enumerate}
 \item[(H2)] $t_{K}<t_{\ast}^{p}\int_{\Omega}e_{1}^{p}dx$;
 \item[(H3)] $\max_{t\in [0, t_{K}]}a(t)<\gamma/\lambda_{1}$;
 \item[(H4)] $\max_{t\in[0, t_{\ast}]}f(t)t_{\ast}^{p-1}<(\lambda_{1}^{1/2}/C_{1}|\Omega|^{1/2})\max_{t\in[t_{k-1}, t_{k}]}a(t)t$,
             for all $k\in\{1, \ldots, K\}$, where $C_{1}$ stands for best constant of the Sobolev embedding from $H_{0}^{1}(\Omega)$ into $L^{1}(\Omega)$.
\end{enumerate}

\medskip

Condition (H3) is trivially verified if $f(0)>0$ (see Remark \ref{rem}). If $\gamma<\infty$, (H3) basically tells us that the peaks of each bump of $a$ are, in some sense,
controlled from above by variation of $f$ at 0. In turn, (H4) means that the peaks of each bump of $a(t)t$ are controlled from below by maximum value of $f$ in
$[0, t_{\ast}]$ (see Fig. 1, where $\theta=(C_{1}/\lambda_{1}^{1/2})t_{\ast}^{p-1}|\Omega|^{1/2}\max_{t\in[0, t_{\ast}]}f(t)$).

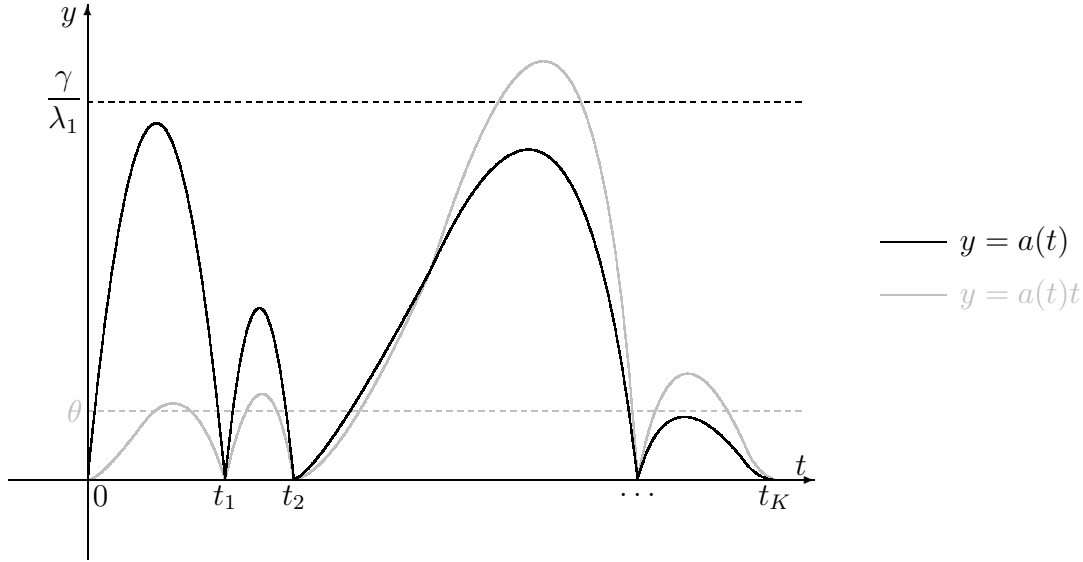
\begin{figure}[h]
\begin{center}
\begin{picture}(430,210)(0,0)
{ \color{Mygrey}
 \bezier{500}(30,30)(35,30)(50,50)
 \bezier{1000}(50,50)(68,75)(82,30)
 \bezier{1500}(82,30)(97,95)(108,30)
 \bezier{1500}(108,30)(126,33)(160,110)
 \bezier{2000}(160,110)(220,300)(238,30)
 \bezier{1000}(238,30)(252,105)(280,40)
 \bezier{00}(280,40)(284,32)(290,30)
 \multiput(30,56)(4,0){68}{\line(2,0){2}}
 \put(28,56){\makebox(0,0)[rc]{{\large $\theta$}}}
 \put(330,100){\line(1,0){25}}\put(360,100){\makebox(0,0)[rl]{{\large $y=a(t)t$}}}
}

 \put(0,30){\vector(1,0){305}} % os pozioma 2
 \put(30,0){\vector(0,1){210}}  % os pionawa 2

 \bezier{2000}(30,30)(56,300)(82,30)
 \bezier{1500}(82,30)(95,160)(108,30)
 \bezier{1500}(108,30)(120,35)(160,110)
 \bezier{2000}(160,110)(215,230)(238,30)
 \bezier{1000}(238,30)(250,75)(280,35)
 \bezier{00}(280,35)(285,30)(290,30)
 \multiput(30,173)(4,0){68}{\line(2,0){2}}
 \put(28,173){\makebox(0,0)[rc]{{\large $\displaystyle\frac{\gamma}{\lambda_1}$}}}
 \put(26,205){\makebox(0,0)[rc]{{\large $y$}}}
 \put(330,120){\line(1,0){25}}\put(360,120){\makebox(0,0)[rl]{{\large $y=a(t)$}}}

 \put(32,28){\makebox(0,0)[lt]{{\large $0$}}}
 \put(82,28){\makebox(0,0)[ct]{{\large $t_1$}}}
 \put(108,28){\makebox(0,0)[ct]{{\large $t_2$}}}
 \put(238,25){\makebox(0,0)[ct]{{\large $\ldots$}}}
 \put(290,28){\makebox(0,0)[ct]{{\large $t_K$}}}
 \put(300,32){\makebox(0,0)[cb]{{\large $t$}}}

\end{picture}
\end{center}
\caption{Geometry of $a(t)$ and $a(t)t$ satisfying $(H1)$, $(H3)$ and $(H4)$.}
\end{figure}

\medskip

The main result of this paper is stated as follows:

\begin{theorem}\label{main}
Suppose \emph{(H0)-(H4)} hold. Then problem \eqref{P} has at least $2K$classical positive solutions with ordered $L^{p}$-norms, namely
$$
0<\int_{\Omega}u_{1, 1}^{p}dx<\int_{\Omega}u_{1, 2}^{p}dx<t_{1}<\ldots<t_{K-1}<\int_{\Omega}u_{K, 1}^{p}dx<\int_{\Omega}u_{K, 2}^{p}dx<t_{K}.
$$
\end{theorem}

Throughout the paper, the following auxiliary problem will play an important role: for each $k\in\{1, \ldots, K\}$ and any $\alpha\in (t_{k-1}, t_{k})$ fixed, consider

\begin{equation}\label{AP}\tag{$P_{k}$}
\left \{ \begin{array}{ll}
-a\left( \alpha\right)\Delta u = f_{\ast}(u) & \mbox{in $\Omega$,}\\
u>0 & \mbox{in $\Omega$,}\\
u=0 & \mbox{on $\partial\Omega$,}
\end{array}\right.
\end{equation}
where
\begin{equation*}
f_{\ast}(t)=\left \{ \begin{array}{ll}
f(0) & \mbox{if $t\leq 0$,}\\
f(t) & \mbox{if $0<t<t_{\ast}$,}\\
0 & \mbox{if $t_{\ast}\leq t$.}\\
\end{array}\right.
\end{equation*}

The paper is organized as follows:\\
In section \ref{se:prelim} we present the proof of the multiplicity result for problem \eqref{P} which is divided into four steps (corresponding to four
subsections): 1) Existence of a unique solution to \eqref{AP} satisfying $0<u_{\alpha}\leq t_{\ast}$; 2) Continuity of the map $(t_{k-1}, t_{k})\ni\alpha\mapsto \mathcal{P}_{k}(\alpha)=\int_{\Omega}u_{\alpha}^{p}dx$; 3) Existence of fixed points for the map $\mathcal{P}_{k}$; 4) Conclusion of the proof of Theorem \ref{main}. In section \ref{se:exam} we provide a way to construct example of functions satisfying our hypotheses.

\medskip

\section{Multiplicity of solutions}\label{se:prelim}

\subsection{Step 1: Existence and uniqueness of a solution to \eqref{AP}}

\medskip

Since we have ``removed'' the nonlocal term of \eqref{P}, we can treat problem \eqref{AP} using a variational approach.

\begin{proposition}\label{pro1}
Suppose \emph{(H0)}, \emph{(H1)}, and \emph{(H3)} hold. Then for each $k\in\{1, \ldots, K\}$ and $\alpha\in (t_{k-1}, t_{k})$ fixed, problem \eqref{AP} has a
unique classical solution $0<u_{\alpha}\leq t_{\ast}$.
\end{proposition}

\begin{proof}
Since $f_{\ast}$ is bounded and continuous, it is standard to prove that the energy functional
$$
I_{k}(u)=a(\alpha)\frac{1}{2}\|u\|^{2}-\int_{\Omega}F_{\ast}(u)dx
$$
of \eqref{AP} is coercive and lower weakly semicontinuous (where $F_{\ast}(s)=\int_{0}^{s}f_{\ast}(\sigma) d\sigma$). Therefore $I_{k}$ has a minimum point
which is a weak solution of \eqref{AP}. Moreover, it follows from (H1) and (H3) that
$$
I_{k}(t\varphi_{1})/t^{2}=\frac{1}{2}a(\alpha)-\int_{\Omega}\frac{F_{\ast}(t\varphi_{1})}{(t\varphi_{1})^{2}}\varphi_{1}^{2}dx\to
\frac{1}{2}\left(a(\alpha)-\frac{\gamma}{\lambda_{1}}\right)<0 \ \mbox{as $t\to 0^{+}$}.
$$
The last inequality implies that any minimum point $u$ of $I_{k}$ is nontrivial because, for $t>0$ small enough
$$
I_{k}(u)\leq I_{k}(t\varphi_{1})=(I_{k}(t\varphi_{1})/t^{2})t^{2}<0.
$$
A simple argument, like in Proposition 3.1 in \cite{SS}, shows that any nontrivial weak solution $u$ of \eqref{AP} satisfies $0\leq u\leq t_{\ast}$. It follows
that \eqref{AP} has a nontrivial weak solution which is unique by (H1) (see \cite{BO}). Since $f_{\ast}(u)=f(u)$ is bounded and $f\in C^{1}(\R)$, it follows
from \cite{Ag} that $u$ is a classical solution. Finally, maximum principle completes the proof (see Theorem 3.1 in \cite{GT}).
\end{proof}

\medskip

\subsection{Step 2: Continuity of the map $\mathcal{P}_{k}$}

\medskip

Next technical lemma will be important to guarantee the continuity of $\mathcal{P}_{k}$. Since, by (H1), the map $(0, t_{\ast})\ni t\mapsto
\psi(t)=f_{\ast}(t)/t$ is decreasing, there exists the inverse, which we will denote by $\psi^{-1}$, and it is defined on $(0, \gamma)$. Thereby, by (H3), for
each $\varepsilon\in (0, \gamma-\lambda_{1}a(\alpha))$, it makes sense to consider function
$$
y_{\alpha}:=\psi^{-1}(\lambda_{1}a(\alpha)+\varepsilon)e_{1}.
$$

\begin{lemma}\label{odd}
Suppose \emph{(H0)}, \emph{(H1)} and \emph{(H3)} hold and let $c_{\alpha}=\inf_{u\in H_{0}^{1}(\Omega)}I_{k}(u)$. Then for each $\varepsilon\in (0,
\gamma-\lambda_{1}a(\alpha))$, we have
\begin{equation}\label{7}
c_{\alpha}\leq -\frac{1}{2}\varepsilon\psi^{-1}(\lambda_{1}a(\alpha)+\varepsilon)^{2}\int_{\Omega}e_{1}^{2} dx, \ \forall \ \alpha\in (t_{k-1}, t_{k}).
\end{equation}
\end{lemma}

\begin{proof}
Observe that by (H1)
$$
F_{\ast}(t)\geq (1/2)f_{\ast}(t)t, \ \forall \ t\geq 0.
$$
Hence,
$$
\frac{I_{k}(y_{\alpha})}{\psi^{-1}(\lambda_{1}a(\alpha)+\varepsilon)^{2}}\leq\frac{1}{2}\left[a(\alpha)\|e_{1}\|^{2}-\int_{\Omega}\frac{f_{\ast}(y_{\alpha})}{\psi^{-1}(\lambda_{1}a(\alpha)+\varepsilon)^{2}}y_{\alpha}dx\right],
$$
or equivalently
$$
\frac{I_{k}(y_{\alpha})}{\psi^{-1}(\lambda_{1}a(\alpha)+\varepsilon)^{2}}\leq\frac{1}{2}\left[a(\alpha)\|e_{1}\|^{2}-\int_{\Omega}\frac{f_{\ast}(y_{\alpha})}{y_{\alpha}}e_{1}^{2}dx\right].
$$
Using the definition of $e_{1}$ and (H1), we get
$$
\frac{I_{k}(y_{\alpha})}{\psi^{-1}(\lambda_{1}a(\alpha)+\varepsilon)^{2}}\leq\frac{1}{2}\left[a(\alpha)\|e_{1}\|^{2}-\int_{\Omega}\frac{f_{\ast}(\psi^{-1}(\lambda_{1}a(\alpha)+\varepsilon))}{\psi^{-1}(\lambda_{1}a(\alpha)+\varepsilon)}e_{1}^{2}dx\right].
$$
Now, using the definition of $\psi^{-1}$, we conclude
$$
\frac{I_{k}(y_{\alpha})}{\psi^{-1}(\lambda_{1}a(\alpha)+\varepsilon)^{2}}=\frac{1}{2}\left[a(\alpha)\|e_{1}\|^{2}-(\lambda_{1}a(\alpha)+\varepsilon)\int_{\Omega}e_{1}^{2}dx\right]=-\frac{1}{2}\varepsilon\int_{\Omega}e_{1}^{2}
dx.
$$
Therefore,
$$
c_{\alpha}\leq I_{k}(y_{\alpha})\leq -\frac{1}{2}\varepsilon\psi^{-1}(\lambda_{1}a(\alpha)+\varepsilon)^{2}\int_{\Omega}e_{1}^{2} dx.
$$
\end{proof}

\begin{proposition}\label{pp2}
Suppose \emph{(H0)}, \emph{(H1)}, and \emph{(H3)} hold. Then for each $k\in \{1, 2, \ldots, K\}$, map $\mathcal{P}_{k}: (t_{k-1}, t_{k})\to \R$ defined by
$$
\mathcal{P}_{k}(\alpha)=\int_{\Omega}u_{\alpha}^{p}dx,
$$
where $p\geq 1$ and $u_{\alpha}$ was obtained in Proposition \ref{pro1}, is continuous.
\end{proposition}

\begin{proof}
Let $\{\alpha_{n}\}\subset (t_{k-1}, t_{k})$ be such $\alpha_{n}\to \alpha_{\ast}$, for some $\alpha_{\ast}\in (t_{k-1}, t_{k})$. Denote by $u_{n}$ the positive solution of $\eqref{AP}$ with $\alpha=\alpha_{n}$. Since,
\begin{equation}\label{1}
\frac{1}{2}a(\alpha_{n})\|u_{n}\|^{2}-\int_{\Omega}F_{\ast}(u_{n})dx=I_{k}(u_{n})<0,
\end{equation}
we get
$$
\|u_{n}\|\leq 2\frac{F_{\ast}(t_{\ast})|\Omega|}{a(\alpha_{n})}, \ \forall \ n\in\N.
$$
Therefore, $\{u_{n}\}$ is bounded in $H_{0}^{1}(\Omega)$ and, up to a subsequence, there exists $u_{\ast}\in H_{0}^{1}(\Omega)$ such that
\begin{equation}\label{19}
u_{n}\rightharpoonup u_{\ast} \ \mbox{in} \ H_{0}^{1}(\Omega).
\end{equation}
Thus, passing to the limit as $n\to\infty$ in
$$
a(\alpha_{n})\int_{\Omega}\nabla u_{n}\nabla v dx=\int_{\Omega} f_{\ast}(u_{n})v dx, \ \forall \ v\in H_{0}^{1}(\Omega),
$$
we get
$$
a(\alpha_{\ast})\int_{\Omega}\nabla u_{\ast}\nabla v dx=\int_{\Omega} f_{\ast}(u_{\ast})v dx, \ \forall \ v\in H_{0}^{1}(\Omega).
$$
So, $u_{\ast}$ is a nonnegative weak solution of \eqref{AP} with $\alpha=\alpha_{\ast}$. We are going to show that $u_{\ast}\neq 0$. In fact, passing to the limit as
$n\to\infty$ in
$$
a(\alpha_{n})\int_{\Omega}\nabla u_{n}\nabla u_{\ast} dx=\int_{\Omega} f_{\ast}(u_{n})u_{\ast} dx
$$
and
$$
a(\alpha_{n})\|u_{n}\|^{2}=\int_{\Omega} f_{\ast}(u_{n})u_{n} dx,
$$
we conclude that
\begin{equation}\label{20}
\|u_{n}\|\to \|u_{\ast}\|.
\end{equation}

By \eqref{19} and \eqref{20},
\begin{equation}\label{21}
u_{n}\to u_{\ast} \ \mbox{in $H_{0}^{1}(\Omega)$}.
\end{equation}

By Lemma \ref{odd}, there exists $\varepsilon>0$, small enough, such that
$$
I_{k}(u_{n})\leq -\frac{1}{2}\varepsilon\psi^{-1}(\lambda_{1}a(\alpha_{n})+\varepsilon)^{2}\int_{\Omega}e_{1}^{2} dx, \ \forall \ n\in\N.
$$
By \eqref{21}, passing to the limit as $n\to\infty$ in the previous inequality, we obtain
$$
I_{k}(u_{\ast})\leq -\frac{1}{2}\varepsilon\psi^{-1}(\lambda_{1}a(\alpha_{\ast})+\varepsilon)^{2}\int_{\Omega}e_{1}^{2} dx<0.
$$
Therefore $u_{\ast}\neq 0$. Arguing as in the proof of Proposition \ref{pro1} we can show that $u_{\ast}$ is a positive classical solution of \eqref{AP} with
$\alpha=\alpha_{\ast}$. Since such a solution is unique, we conclude that $u_{\ast}=u_{\alpha_{\ast}}$. Consequently,
\begin{equation}\label{21b}
-\Delta(u_{n}-u_{\ast})=\frac{a(\alpha_{n})-a(\alpha_{\ast})}{a(\alpha_{n})}\Delta u_{\ast}+\frac{f_{\ast}(u_{n})-f_{\ast}(u_{\ast})}{a(\alpha_{n})}=:g_{n}(x), \ \forall \ n\in\N.
\end{equation}
Since $f_{\ast}$ is bounded and $a(\alpha_{n})$ is away from zero, there exists a positive constant $C$, such that
\begin{equation}\label{22}
|g_{n}|_{\infty}\leq C, \ \forall \ n\in\N.
\end{equation}
It follows follows \eqref{21b}, \eqref{22} and Theorem 0.5 in \cite{AP} that there exists $\beta\in (0, 1)$ such that
$$
\|u_{n}-u_{\ast}\|_{C^{1,\beta}(\overline{\Omega})}\leq C,  \ \forall \ n\in\N,
$$
for some $C>0$. By the compactness of embedding from $C^{1,\beta}(\overline{\Omega})$ into $C^{1}(\overline{\Omega})$, up to a
subsequence, we have
\begin{equation}\label{23}
u_{n}\to u_{\ast} \ \mbox{in $C^{1}(\overline{\Omega})$}.
\end{equation}
Convergence in \eqref{23} and inequality
$$
||u_{n}|_{p}-|u_{\ast}|_{p}|\leq |u_{n}-u_{\ast}|_{p}\leq |\Omega|^{1/p}|u_{n}-u_{\ast}|_{\infty}
$$
lead us to
$$
\mathcal{P}_{k}(\alpha_{n})\to\mathcal{P}_{k}(\alpha_{\ast}).
$$
This proves the continuity of $\mathcal{P}_{k}$.

\end{proof}

\medskip

\subsection{Step 3: Existence of fixed points}

\medskip

Next lemma will be helpful in obtaining fixed points of $\mathcal{P}_{k}$.

\begin{lemma}\label{odt}
Suppose \emph{(H0)}, \emph{(H1)} and \emph{(H3)} hold. Then
\begin{equation}\label{2}
u_{\alpha}\geq z_{\alpha}:=\psi^{-1}(\lambda_{1}a(\alpha))e_{1}, \ \forall \ \alpha\in (t_{k-1}, t_{k}).
\end{equation}
\end{lemma}

\begin{proof}

In fact, it follows from (H1) and the definition of $\psi^{-1}$ that
$$
\lambda_{1}a(\alpha)= \frac{f_{\ast}(\psi^{-1}(\lambda_{1}a(\alpha)))}{\psi^{-1}(\lambda_{1}a(\alpha))}\leq \frac{f_{\ast}(z_{\alpha})}{z_{\alpha}}.
$$
Thus
$$
-a(\alpha)\Delta(z_{\alpha})=\lambda_{1}a(\alpha)z_{\alpha}\leq f_{\ast}(z_{\alpha}) \ \mbox{in $\Omega$}.
$$

Therefore $z_{\alpha}$ is a subsolution of \eqref{AP}. Inequality \eqref{2} follows now from (H1) and Lemma 3.3 in \cite{ABC}.
\end{proof}

\medskip

\begin{proposition}\label{pp3}
Suppose \emph{(H0)-(H4)} hold. Then map $\mathcal{P}_{k}$ has at least two fixed points $t_{k-1}<\alpha_{1, k}<\alpha_{2, k}<t_{k}$.
\end{proposition}

\begin{proof}

We start with two claims describing the geometry of $\mathcal{P}_{k}$ (see Fig. 2).

\begin{figure}[h]
\begin{center}
\begin{picture}(300,280)(0,0)
 \put(0,30){\vector(1,0){280}} % os pozioma 2
 \put(30,0){\vector(0,1){280}}  % os pionawa 2
% \put(10,10){\line(1,1){260}}
  \multiput(10,10)(2,2){130}{\circle*{0.01}}

%asymptots
% \put(0,240){\line(1,0){260}}
% \multiput(30,240)(4,0){62}{\line(2,0){1}}
% \put(80,30){\line(0,1){240}}
 \multiput(80,30)(0,4){62}{\line(0,2){1}}
% \put(200,30){\line(0,1){240}}
 \multiput(200,30)(0,4){62}{\line(0,2){1}}
% \put(240,30){\line(0,1){240}}
% \multiput(240,30)(0,4){62}{\line(0,2){1}}

% \put(104,30){\line(0,1){76}}
 \multiput(104,30)(0,2){38}{\line(0,1){1}}
    \put(104,104){\circle*{3}}
% \put(174,30){\line(0,1){144}}
 \multiput(174,30)(0,2){72}{\line(0,1){1}}
    \put(174,174){\circle*{3}}

 \bezier{3000}(84,270)(100,-100)(200,260)
   \color{white}\put(200,260){\circle*{3}}
   \color{black}\put(200,260){\circle{3}}

% opis osi poziomej
  \put(80,26){\makebox(0,0)[ct]{{\large $t_{k-1}$}}}
  \put(104,24){\makebox(0,0)[ct]{{\large $\alpha_{1,k}$}}}
  \put(174,24){\makebox(0,0)[ct]{{\large $\alpha_{2,k}$}}}
  \put(202,26){\makebox(0,0)[ct]{{\large $t_k$}}}
%  \put(242,26){\makebox(0,0)[ct]{{\large $t_K$}}}
  \put(268,34){\makebox(0,0)[lb]{{\large $\alpha$}}}

% opis osi pionowej
  \put(25,268){\makebox(0,0)[rc]{{\large $y$}}}
%  \put(27,240){\makebox(0,0)[rc]{{\large $t_K$}}}

%opis rysunku
  \put(140,268){\makebox(0,0)[rc]{{\large $y=P_k(\alpha)$}}}
  \put(295,260){\makebox(0,0)[rc]{\large $y=\alpha$}}

\end{picture}

\end{center}
\caption{Geometry of $\mathcal{P}_{k}$ in $(t_{k-1},t_k)$.}
\end{figure}
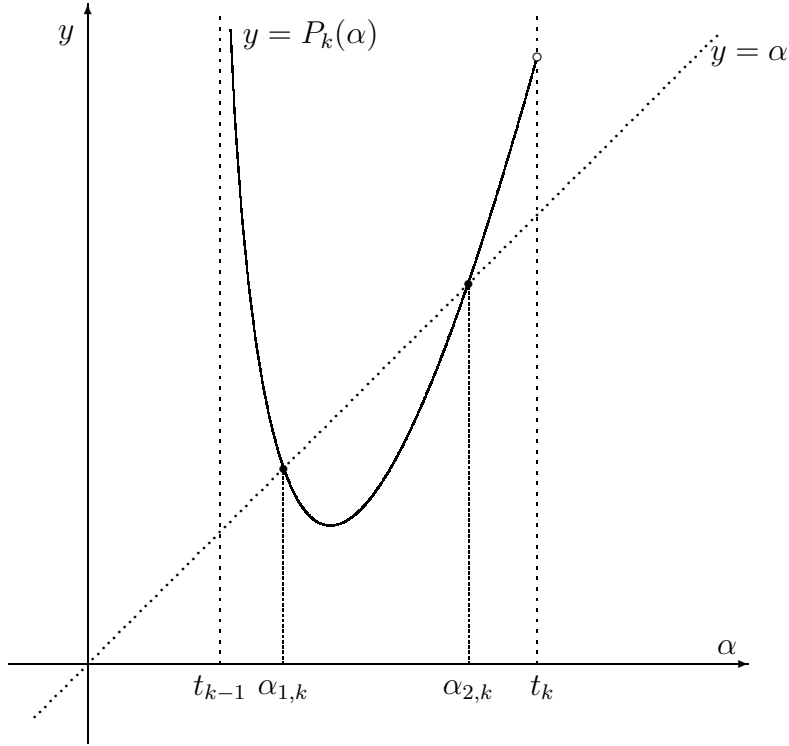

\noindent {\bf Claim 1:} $\lim\limits_{\alpha\to t_{k-1}^{+}}\mathcal{P}_{k}(\alpha)>t_{k-1}$ and $\lim\limits_{\alpha\to
t_{k}^{-}}\mathcal{P}_{k}(\alpha)>t_{k}$.

\medskip

From Lemma \ref{odt}, we have
$$
\mathcal{P}_{k}(\alpha)\geq (\psi^{-1}(\lambda_{1}a(\alpha)))^{p}\int_{\Omega}e_{1}^{p}dx, \ \forall \ \alpha\in (t_{k-1}, t_{k}).
$$
Hence, by (H2)
$$
\lim_{\alpha\to t_{k-1}^{+} \ \mbox{or} \ t_{k}^{-}}\mathcal{P}_{k}(\alpha)\geq t_{\ast}^{p}\int_{\Omega}e_{1}^{p}dx>t_{K}>t_{k}>t_{k-1}.
$$

\medskip

\noindent {\bf Claim 2:} There exists $\alpha\in (t_{k-1}, t_{k})$ such that $\mathcal{P}_{k}(\alpha)<\alpha$.

\medskip

For each $\alpha\in (t_{k-1}, t_{k})$, let $w_{\alpha}$ be the unique solution (which is positive) of the problem
\begin{equation*}
\left \{ \begin{array}{ll}
-\Delta u = u_{\alpha}^{p-1} & \mbox{in $\Omega$,}\\
u=0 & \mbox{on $\partial\Omega$,}
\end{array}\right.
\end{equation*}
where $u_{\alpha}$ is the unique positive solution of \eqref{AP}. Hence, multiplying by $u_{\alpha}$ and integrating by parts, we have
$$
\int_{\Omega}\nabla w_{\alpha}\nabla u_{\alpha} dx=\int_{\Omega}u_{\alpha}^{p}dx=\mathcal{P}_{k}(\alpha).
$$
On the other hand, by using the definition of $u_{\alpha}$, we get
\begin{equation}\label{12}
\mathcal{P}_{k}(\alpha)=\frac{1}{a(\alpha)}\int_{\Omega}f_{\ast}(u_{\alpha})w_{\alpha} dx.
\end{equation}

By definition of $w_{\alpha}$, the fact that $0<u_{\alpha}\leq t_{\ast}$ and H\"older's inequality, we obtain
\begin{equation}\label{15}
\|w_{\alpha}\|\leq  (1/\lambda_{1}^{1/2})\left(\int_{\Omega}u_{\alpha}^{2(p-1)}dx\right)^{1/2}\leq (1/\lambda_{1}^{1/2})t_{\ast}^{p-1}|\Omega|^{1/2}.
\end{equation}
Thus,
\begin{equation}\label{16}
\mathcal{P}_{k}(\alpha)\leq\frac{1}{a(\alpha)}\left(\max_{[0, t_{\ast}]}f(t)\right)C_{1}\|w_{\alpha}\|,
\end{equation}
where $C_{1}>0$ is the best constant of the Sobolev embedding from $H_{0}^{1}(\Omega)$ into $L^{1}(\Omega)$. Applying \eqref{15} in \eqref{16}, we obtain
$$
\mathcal{P}_{k}(\alpha)\leq \frac{1}{a(\alpha)}\left(\max_{[0, t_{\ast}]}f(t)\right)(C_{1}/\lambda_{1}^{1/2})t_{\ast}^{p-1}|\Omega|^{1/2}, \ \forall \
\alpha\in (t_{k-1}, t_{k}).
$$
Using (H4) we get the conclusion of Claim 2.

\medskip

The proof follows from Proposition \ref{pp2}, Claim 1, Claim 2 and the intermediate value theorem for continuous real functions.
\end{proof}

\subsection{Step 4: Proof of Theorem \ref{main}}

\medskip

For each fixed $k\in \{1, \ldots, K\}$, it follows from Propositions \ref{pro1} and \ref{pp3} that \eqref{P} has two classical positive solutions $u_{k, 1}$
and $u_{k, 2}$ such that
$$
t_{k-1}<\int_{\Omega}u_{k, 1}^{p}dx<\int_{\Omega}u_{k, 2}^{p}dx<t_{k}.
$$
This finishes the proof.\hspace*{\fill} $\square$

%First, notice that there exist $\theta>0$ such that
%$$
%\int_{\Omega}f_{\ast}(u_{\alpha})w_{\alpha} dx\geq \theta, \ \forall \ \alpha\in (t_{k-1}, t_{k}).
%$$
%Indeed, otherwise, for some subsequence $u_{n}:=u_{\alpha_{n}}$ (and $w_{n}:=w_{\alpha_{n}}$) we have
%$$
%\int_{\Omega}f_{\ast}(u_{n})w_{n} dx\to 0.
%$$
%Consequently,
%
%
%$$
%u_{n}\rightharpoonup u_{\ast} \ \mbox{in} \ H_{0}^{1}(\Omega),
%$$
%for some $u\in H_{0}^{1}(\Omega)$. Thereby,
%$$
%c_{n}=I_{k}(u_{n})=\frac{1}{2}a(\alpha_{n})\|u_{n}\|^{2}-\int_{\Omega}f_{\ast}(u_{n})dx\to -\int_{\Omega}f_{\ast}(u_{\ast}) dx
%$$

\section{Example}\label{se:exam}

  We provide an example of functions $a$ and $f$ satisfying hypotheses (H0)-(H4).

  Let $0=t_0<t_1<\ldots<t_K$ and $t_{\ast}$ be such that $(H2)$ holds.
  Let $a\colon [0,t_K]\to\mathbb{R}$ be any function satisfying (H0) and (H3) with some $\gamma>0$.
  Denote
\[
  A:=\min_{k\in\{1,\ldots,K\}}\max_{t\in [t_{k-1},t_k]}a(t)t>0\quad\textrm{and}\quad M:=\lambda_{1}^{1/2}A/C_{1}|\Omega|^{1/2}t_{\ast}^{p-1}.
\]
  Choose any $\eta>\max\{\gamma/M,\ t_{\ast}/M,\ 1/t_{\ast}+1/\gamma\}$ and fix
  $c:=\eta^2\gamma-(\gamma/t_{\ast}+1)\eta$
  (note that $c>0$ from the choice of $\eta$).
  Let
\[
  f(t):=\gamma t\cdot\frac{1-t/t_{\ast}}{1+ct}.
\]
  We need to check that $f$ satisfies all the assumptions.
  First note that the map $t\mapsto f(t)/t$ is decreasing on $[0, t_{\ast}]$,
  $\lim_{t\to 0^{+}}f(t)/t=\gamma$, $f(0)=f(t_{\ast})=0$ and $f(t)>0$ for all $t\in (0,t_{\ast})$
  (so (H0) and (H1) hold).
  Finally note that
\[
  f(t)\le \gamma\cdot 1/\eta<M\quad\forall t\in [0,1/\eta]
\]
  and
\[
  f(t)=t\cdot\frac{f(t)}{t}\le t_{\ast}\cdot1/\eta<M\quad\forall t\in [1/\eta,t_{\ast}]
\]
  (as $f(t)/t$ is decreasing and $f(1/\eta)/(1/\eta)=1/\eta$).
  Thus $\max_{t\in[0, t_{\ast}]}f(t)<M$ and so (H4) holds.

\end{document}